\documentclass[11pt]{amsart}

\usepackage{amscd}
\usepackage{amsmath, amssymb}
\usepackage{amsfonts}
\newcommand{\de}{\partial}
\newcommand{\db}{\overline{\partial}}

\newcommand{\ddbar}{i \partial \overline{\partial}}

\newcommand{\ov}[1]{\overline{#1}}

\newcommand{\ti}[1]{\tilde{#1}}
\newcommand{\vp}{\varphi}
\newcommand{\vol}{\mathrm{Vol}}

\newcommand{\ve}{\varepsilon}

\numberwithin{equation}{section}
\renewcommand{\leq}{\leqslant}
\renewcommand{\geq}{\geqslant}

\begin{document}
\newtheorem{claim}{Claim}
\newtheorem{theorem}{Theorem}[section]
\newtheorem{conjecture}[theorem]{Conjecture}
\newtheorem{lemma}[theorem]{Lemma}
\newtheorem{corollary}[theorem]{Corollary}
\newtheorem{proposition}[theorem]{Proposition}
\newtheorem{question}[theorem]{question}
\newtheorem{conj}[theorem]{Conjecture}
\newtheorem{defn}[theorem]{Definition}
\theoremstyle{definition}
\newtheorem{remark}[theorem]{Remark}
     \newtheorem{hypothesis}[theorem]{Hypothesis}

\newenvironment{example}[1][Example]{\addtocounter{remark}{1} \begin{trivlist}
\item[\hskip
\labelsep {\bfseries #1  \thesection.\theremark}]}{\end{trivlist}}

\title{Transcendental Morse inequality on K\"ahler manifolds}

    \author{Valentino Tosatti}

\address{Courant Institute School of Mathematics, Computing, and Data Science, New York University, 251 Mercer St, New York, NY 10012}
\email{tosatti@cims.nyu.edu}

\begin{abstract}
We prove the transcendental Morse inequality conjecture of Boucksom-Demailly-P\u{a}un-Peternell. Among the corollaries of this result, this proves orthogonality of divisorial Zariski decompositions, shows that the pseudoeffective and movable cones are dual, and establishes the differentiability of the volume function on the big cone, on all compact K\"ahler manifolds.
\end{abstract}

\maketitle

\section{Introduction}

\subsection{The volume function} The volume $\vol([\alpha])\in\mathbb{R}_{\geq 0}$ of a $(1,1)$ cohomology class $[\alpha]\in H^{1,1}(X,\mathbb{R})$ on a compact K\"ahler manifold $X^n$ is a fundamental quantity, introduced by Boucksom \cite{Bou}, which specializes to the volume of a holomorphic line bundle $L$ in the case when $[\alpha]=c_1(L)$:
\begin{equation}
\vol(L)=\limsup_{k\to\infty} \frac{h^0(X,kL)}{k^n/n!}.
\end{equation}
For a general $(1,1)$-class $[\alpha]$, its volume is set to be zero if the class is not pseudoeffective (i.e. it does not contain any closed positive $(1,1)$-current), and otherwise it is defined by
\begin{equation}
\vol([\alpha])=\sup_{T\in[\alpha], T\geq 0}\int_X T_{\rm ac}^n,
\end{equation}
where the supremum ranges over all closed positive $(1,1)$-currents in the class $[\alpha]$, and $T_{\rm ac}$ denotes the absolutely continuous part in the Lebesgue decomposition of $T$. Boucksom \cite{Bou} shows that the resulting function $\vol:H^{1,1}(X,\mathbb{R})\to\mathbb{R}_{\geq 0}$ is continuous, and that classes with strictly positive volume (``big'') are characterized by the fact that they contain a K\"ahler current (a closed positive current which dominates some K\"ahler form).

\subsection{Transcendental Morse inequality}
The main result of this paper is a resolution of the following ``transcendental Morse inequality'' conjecture about the volume function, which was formulated by Boucksom-Demailly-P\u{a}un-Peternell \cite{BDPP} in 2004, see also \cite{Dem3}:
\begin{theorem}\label{main}
Let $X^n$ be a compact K\"ahler manifold, and let $[\alpha],[\beta]\in H^{1,1}(X,\mathbb{R})$ be two nef classes. Then we have
\begin{equation}\label{vol}
\vol([\alpha]-[\beta])\geq \int_X\alpha^n-n\int_X \alpha^{n-1}\wedge \beta.
\end{equation}
\end{theorem}
The case when $X$ is projective and $[\alpha],[\beta]$ are first Chern classes of line bundles is a classical result due to Siu, Catanese, Ein-Lazarsfeld (see \cite[Thm. 2.2.15]{Laz}), and independently Trapani \cite{Tr}, and this easily extends to $\mathbb{R}$-divisors \cite[Lemma 4.2]{BDPP}. Further analytic and algebraic refinements were found by Demailly \cite[\S 12]{Dem4}, Trapani \cite{Tr2} and Angelini \cite{An}. The case of arbitrary nef $(1,1)$-classes on a projective manifold was proved by Witt Nystr\"om \cite{WN} in 2016. For general K\"ahler manifolds, Popovici \cite{Po} (sharpening ideas of Chiose \cite{Ch} and Xiao \cite{Xi}) proved that if the RHS of \eqref{vol} is strictly positive, then so is the LHS. His method also proves a weaker bound than \eqref{vol}, with the constant $n$ on the RHS replaced by $n^2$ \cite{To, Po2}.
\subsection{Applications}
The seemingly innocuous statement of Theorem \ref{main}, which is the transcendental version of a relatively easy statement in algebraic geometry, turns out to have far-reaching consequences which were discovered in the pioneering and deep works of Boucksom, Demailly, Favre, Jonsson, P\u{a}un and Peternell in \cite{BDPP,BFJ}. Indeed, as explained for example in the clear exposition by Boucksom in the appendix of \cite{WN}, the following results are corollaries of Theorem \ref{main}:
\begin{corollary}\label{cor}
Let $X^n$ be a compact K\"ahler manifold. Then we have:
\begin{itemize}
\item (Orthogonality of divisorial Zariski decompositions) For every pseudoeffective class $[\alpha]\in H^{1,1}(X,\mathbb{R})$,
\begin{equation}\label{orth}
\vol([\alpha])=\langle [\alpha]^{n-1}\rangle\cdot[\alpha].
\end{equation}
\item (Differentiability of $\vol$) $\vol$ is locally $C^{1,1}$ in the cone of big $(1,1)$-classes, and its gradient can be described as follows. Given $[\alpha],[\beta]\in H^{1,1}(X,\mathbb{R})$ with $[\alpha]$ big, we have
\begin{equation}\label{deriv}
\frac{d}{dt}\bigg|_{t=0}\vol([\alpha]+t[\beta])=n\langle [\alpha]^{n-1}\rangle\cdot[\beta].
\end{equation}
\item (Dual of the pseudoeffective cone) The pseudoeffective cone $\mathcal{E}\subset H^{1,1}(X,\mathbb{R})$ and the movable cone $\mathcal{M}\subset H^{n-1,n-1}(X,\mathbb{R})$ are dual under the Poincar\'e pairing.
\end{itemize}
\end{corollary}
Here $\langle \cdot\rangle$ denotes the movable intersection products of Boucksom \cite{BoT,BDPP}, and the movable cone $\mathcal{M}$ is the closed cone generated by classes of $(n-1,n-1)$-currents of the form $\pi_*(\ti{\omega_1}\wedge\cdots\wedge\ti{\omega}_{n-1})$ where $\pi:\ti{X}\to X$ is any modification and $\ti{\omega}_j$ are K\"ahler forms on $\ti{X}$.

A few remarks about these results. Divisorial Zariski decompositions were defined by Boucksom \cite{Bo2}, and the orthogonality statement in \eqref{orth} can be interpreted as the appropriate generalization of the fact that on surfaces the positive and negative parts of the Zariski decomposition of a line bundle are orthogonal, see the discussion in \cite{BDPP}. Boucksom-Demailly-P\u{a}un-Peternell \cite{BDPP} had proved orthogonality and the duality statement for projective manifolds and classes in the real N\'eron-Severi group, and this last restriction was removed by Witt Nystr\"om. In the general setting, the author \cite{To2} proved orthogonality \eqref{orth} when $\vol([\alpha])=0$.

The $C^1$ differentiability of $\vol$ on the N\'eron-Severi part of the big cone of a projective manifold was proved independently by Boucksom-Favre-Jonsson \cite{BFJ} and Lazarsfeld-Musta\c{t}\u{a} \cite{LM}. Boucksom-Favre-Jonsson also proved \eqref{deriv}, and showed that their proof would go through in the general case provided Theorem \ref{main} holds, see again the appendix of \cite{WN}. Cao and the author \cite{CaT} recently observed that the transcendental Morse inequality \eqref{vol}, together with the $C^1$ differentiability and \eqref{deriv}, implies that the volume function is locally $C^{1,1}$ inside the big cone of a compact K\"ahler manifold, and this is sharp in general. This result is now therefore proved unconditionally. At the boundary of the big cone, $\vol$ is only locally Lipschitz and no better, in general, see \cite{CaT,FLT} for examples, even when restricting $\vol$ to segments.

Furthermore, if we take $[\beta]=c_1(D)$ where $D\subset X$ is a prime divisor, then Witt Nystr\"om \cite{WN2} and Vu \cite{Vu} proved unconditionally that the derivative in \eqref{deriv} exists, and equals $n\langle [\alpha]^{n-1}\rangle|_{X|D},$ which is $n$ times the numerical restricted volume of $[\alpha]$ along $D$, as defined by Boucksom \cite{BoT} and studied further by Collins and the author \cite{CT2}. As a corollary of this (or following directly the arguments in \cite{BFJ}), we obtain the transcendental generalization of \cite[Thm B]{BFJ}:
\begin{corollary}
If $D\subset X$ is a prime divisor on a compact K\"ahler manifold and $[\alpha]\in H^{1,1}(X,\mathbb{R})$ is a big class, then we have
\begin{equation}\label{uguali}
\langle [\alpha]^{n-1}\rangle|_{X|D}=\langle [\alpha]^{n-1}\rangle \cdot D.
\end{equation}
Furthermore, this quantity equals zero if and only if $D$ is contained in the non-K\"ahler locus $E_{nK}([\alpha])$.
\end{corollary}
The last statement, which is the divisorial part of a conjecture of Collins and the author \cite{CT2}, was already proved in \cite{Vu,WN2}, but \eqref{uguali} was not known in general.

The last corollary that we will mention follows immediately from the work of Toma \cite{Tom}, as observed by Fu-Xiao \cite[Appendix]{FX}, and the arguments in \cite[Remark 5.2]{BDPP}, together with Theorem \ref{main}:
\begin{corollary}
Let $X$ be a compact K\"ahler manifold. Then the movable cone $\mathcal{M}\subset H^{n-1,n-1}(X,\mathbb{R})$ equals the closure of the balanced cone, which consists of cohomology classes of the form $[\gamma^{n-1}]$ where $\gamma$ is a Hermitian metric with $d(\gamma^{n-1})=0$ (``balanced''). Furthermore, $\mathcal{M}$ also equals the closed cone generated by classes of $(n-1,n-1)$-currents of the form $\pi_*(\ti{\omega}^{n-1})$ where $\pi:\ti{X}\to X$ is any modification and $\ti{\omega}$ is a K\"ahler form on $\ti{X}$.
\end{corollary}

\subsection{Ideas of proof}
The basic method that we employ to prove Theorem \ref{main}, which was also underpinning many of the above-mentioned previous works, is the {\em mass concentration} technique, which first appeared in Demailly's visionary paper \cite{Dem2}, and was later greatly extended by Demailly-P\u{a}un \cite{DP}, and further refined in \cite{Ch, Xi, Po}. By definition, to obtain a lower bound for the volume of $[\alpha]-[\beta]$, it suffices to construct a closed positive current $T$ in this class with an explicit lower bound for its absolutely-continuous Monge-Amp\`ere operator $T_{\rm ac}^n$, which after integration provides a lower bound for the volume. For this, Demailly's idea is to solve a family of complex Monge-Amp\`ere equations, whose solutions are provided by the Calabi-Yau Theorem \cite{Ya}, and with RHS of the equation that degenerates in the limit, to a singular measure that concentrates along a closed analytic subvariety. One then shows that any weak limit $T$ of these solutions will have a nontrivial mass along the subvariety, and one can try to leverage this to bound $T_{\rm ac}^n$. In the case of \eqref{vol}, this strategy was first implemented in \cite[Appendix]{BDPP} in the projective case, and resulted in a suboptimal constant. The works \cite{Ch, Xi, Po} observed that the mass concentration technique can be coupled with a duality result of Lamari \cite{Lam}, involving Gauduchon metrics (Hermitian metrics $\chi$ with $\de\db(\chi^{n-1})=0$), to produce the desired current $T$ via the Hahn-Banach theorem (a versatile technique discovered by Sullivan \cite{Su}). Subsequently, Witt Nystr\"om exploited the fact that in the projective case one may assume $[\beta]=c_1(A)$ where $A$ is very ample, and use the current of integration $[H]$ along a divisor $H\in |A|$ to ``subtract off'' $[\beta]$ from $[\alpha]$ and relate directly the space of $\alpha$-psh functions and of $(\alpha-\beta)$-psh functions. He also remarked, see also McCleerey \cite{Mc}, that in the general K\"ahler case, the role of $[H]$ could be played by a closed positive current in a K\"ahler class $[\beta]$ which is supported on a pluripolar set. Unfortunately, it seems very hard to construct such plurisupported currents in general.

Our first observation is that it is easy to construct closed positive currents in any K\"ahler class which are supported on a Lebesgue nullset (Proposition \ref{conc}). These will typically be supported on a union of real-analytic hypersurfaces. We then modify the mass concentration procedure of \cite{BDPP} using one of these currents (or rather its regularizations) in the class $[\beta]$ (which may be assumed to be K\"ahler), and are able to show that any weak limit $\gamma_\infty$ of the solutions of the family of complex Monge-Amp\`ere equations \eqref{ma}, which belong to $[\alpha]$, will have a sharp lower bound for $(\gamma_\infty)_{\rm ac}^n$. Here we use crucially that our auxiliary current is supported on a nullset, which forces the singular part $(\gamma_\infty)_{\rm sing}$ of the limit to have at least the same mass as what a current in the class $[\beta]$ would have. This can be viewed as a transcendental analog of the above-mentioned subtraction of $[\beta]$ in \cite{WN}. In this argument, there is also an auxiliary Gauduchon metric $\chi$, and the output of the mass concentration is then fed into a modification of the Lamari's Hahn-Banach argument (Proposition \ref{hb}), which instead of producing a current $R$ (in the class $[\alpha]-[\beta]$) with an explicit positive lower bound (which is what the arguments in \cite{Lam,Ch,Xi,Po,To} do), gives an explicit positive lower bound for $R_{\rm ac}^n$, which is exactly what suffices to estimate $\vol([\alpha]-[\beta])$.

\subsection*{Acknowledgments} The author is grateful to David Witt Nystr\"om and Nick McCleerey for discussions around 2018 about the method of \cite{WN} and plurisupported currents. The idea that the role of plurisupported currents in a K\"ahler class (whose existence remains unclear) could be played by currents supported on a Lebesgue nullset (which can be easily constructed) was discovered during discussions with ChatGPT. The author also thanks Mihai P\u{a}un for comments on a previous version. The author was partially supported by NSF grant DMS-2404599.

\section{Proof of the main theorem}
Throughout the paper we assume $n\geq 2$, since when $n=1$ the main theorem is well-known, simple, and left as an exercise.
\subsection{Currents supported on a nullset}
\begin{proposition}\label{conc} Let $(X^n,\omega)$ be a compact K\"ahler manifold. Then there is a closed positive current $T\in [\omega]$ with
\begin{equation}\label{ac}
\mathrm{Supp}(T)\subset\Sigma,
\end{equation}
where $\Sigma\subset X$ is closed with Lebesgue measure zero.
\end{proposition}
\begin{proof}
Given any $x\in X$, we can find an open set $x\in U$ contained in a coordinate chart, and a smooth strictly plurisubharmonic function $\vp$ on $U$ such that $\omega=\ddbar \vp$ on $U$. As in \cite[Proof of Thm. 4]{Blo}, up to adding to $\vp$ the real part of a holomorphic quadratic polynomial, we may assume without loss that $\vp$ attains a strict local minimum at $x$. Up to shrinking $U$, there is $a>0$ such that $\min_{\de U}\vp\geq \vp(x)+a$, and we may take $U$ to be a sufficiently small geodesic ball with respect to some fixed Riemannian metric. Thus, for any $c\in (\vp(x),\vp(x)+a/2)$, the open set $\{\vp<c\}$ is compactly contained in $U$ and contains $x$. Thus,
\begin{displaymath}
   \Phi = \left\{
     \begin{array}{ll}
     \max\{c-\vp,0\} &  \text{ on }U,\\
      0 &  \text{ on } X\backslash U,
     \end{array}
   \right.
\end{displaymath}
defines a continuous nonnegative $\omega$-psh function on $X$, with $\omega+\ddbar\Phi$ identically zero on the open neighborhood $\{\vp<c\}$ of $x$ (where it is of course smooth). We will use the freedom to pick the constant $c$ in an open interval.

Repeating this construction at every point $x\in X$, and using compactness, we obtain an open cover $X=\bigcup_{i=1}^N U_i$ by open sets as above, with points $x_i\in U_i$, smooth functions $\vp_i$ on $U_i$, constants $a_i>0$, such that if we pick $c_i\in(\vp_i(x_i),\vp_i(x_i)+a_i/2)$, then the continuous function
\begin{displaymath}
   \Phi_i = \left\{
     \begin{array}{ll}
     \max\{c_i-\vp_i,0\} &  \text{ on }U_i,\\
      0 &  \text{ on } X\backslash U_i,
     \end{array}
   \right.
\end{displaymath}
is $\omega$-psh, nonnegative, supported on $U_i$, smooth on $\{\Phi_i>0\}\ni x_i$ and with $\omega+\ddbar\Phi_i\equiv 0$ there. Each nonempty intersection $U_i\cap U_j$ is connected. Up to replacing $(\vp_i(x_i),\vp_i(x_i)+a_i/2)$ with a smaller subinterval, we may assume that the open sets $\{\Phi_i>0\}$ still cover $X$, whenever the constants $c_i$ are chosen in these subintervals.

We now make the choice of the constants $c_i$ as follows. For each $1\leq i<j\leq N$ with $U_i\cap U_j\neq\emptyset$, we have that $\vp_i-\vp_j$ is pluriharmonic on $U_i\cap U_j$ (which is connected), hence real analytic. If it happens that $(\vp_i-\vp_j)|_{U_i\cap U_j}$ is a constant, say $A$, then we will choose our constants so that $c_i-c_j\neq A$. Varying the indices $i,j$, we have only finitely many such conditions to satisfy, and since each constant $c_i$ can be chosen in some open subinterval of $(\vp_i(x_i),\vp_i(x_i)+a_i/2)$, it follows that we can make a choice for the constants $\{c_i\}_{i=1}^N$ which satisfies all these constraints.

With this choice fixed, we obtain the functions $\Phi_i$ as above (with the sets $\{\Phi_i>0\}$ covering $X$), and we set $\Phi:=\max_{1\leq i\leq N}\Phi_i$. This is a continuous function on $X$, which is $\omega$-psh, and we claim that the closed positive real $(1,1)$-current $T:=\omega+\ddbar\Phi$ satisfies the requirement \eqref{ac}. It is clear that $\Phi>0$ on $X$. For each $1\leq i<j\leq N,$ define closed sets
\begin{equation}
\Sigma_{ij}:=\{\Phi_i=\Phi_j=\Phi\}\subset U_i\cap U_j,
\end{equation}
and
\begin{equation}
\Sigma:=\bigcup_{1\leq i<j\leq N}\Sigma_{ij}.
\end{equation}
We claim that $\Sigma$ has Lebesgue measure zero. Indeed, given $1\leq i<j\leq N$ with $U_i\cap U_j\neq \emptyset$, if $x\in \Sigma_{ij}$ then we must have
\begin{equation}
\vp_i(x)<c_i,\quad \vp_j(x)<c_j,
\end{equation}
and so
\begin{equation}
\Phi_i(x)=c_i-\vp_i(x), \quad \Phi_j(x)=c_j-\vp_j(x),
\end{equation}
which shows that
\begin{equation}
\Sigma_{ij}\subset\{\vp_i-\vp_j=c_i-c_j\}.
\end{equation}
If $\vp_i-\vp_j$ is constant on $U_i\cap U_j$, then this constant cannot be $c_i-c_j$ thanks to our choice, so in this case $\{\vp_i-\vp_j=c_i-c_j\}$ is empty.
If, on the other hand, $\vp_i-\vp_j$ is nonconstant on $U_i\cap U_j$, then the level set $\{\vp_i-\vp_j=c_i-c_j\}$ has Lebesgue measure zero since $\vp_i-\vp_j$ is real analytic. This proves that $\Sigma$ is a Lebesgue nullset.

Finally, we show that $\mathrm{Supp}(T)\subset\Sigma$. Given $x\in X\backslash\Sigma$, by definition there is a unique index $1\leq j\leq N$ such that $\Phi(x)=\Phi_j(x)$. Since the functions $\Phi_i$ are continuous, and $\Phi_i(x)\neq \Phi(x)$ for all $i\neq j$, it follows that the equality $\Phi=\Phi_j$ holds on an open neighborhood $U$ of $x$. Since $\Phi_j(x)=\Phi(x)>0$, up to shrinking $U$ we may assume that $\Phi_j>0$ on $U$, hence on $U$ we have
\begin{equation}
T=\omega+\ddbar\Phi=\omega+\ddbar\Phi_j\equiv 0,
\end{equation}
as desired.
\end{proof}
\begin{remark}
The support of the currents $T$ produced by Proposition \ref{conc} will typically be a union of real analytic hypersurfaces, and will not be pluripolar in general.
\end{remark}

\subsection{The Hahn-Banach argument}
Given a closed positive real $(1,1)$-current $T$, there is a Lebesgue decomposition $T=T_{\rm ac}+T_{\rm sing}$, see \cite{Bou}, into a positive current whose coefficient measures are absolutely continuous with respect to Lebesgue, and a positive current whose coefficients are singular. This decomposition is linear in $T$.

\begin{proposition}\label{hb}
Let $(X^n,\omega)$ be a compact K\"ahler manifold, $\psi$ be a closed real $(1,1)$-form, and $f>0$ a smooth strictly positive function. Suppose that given any Gauduchon metric $\chi$ on $X$, there is a positive $(1,1)$-current $\xi$ on $X$ (that may depend on $\chi$) such that
\begin{equation}\label{1}
\int_X\psi\wedge\chi^{n-1}\geq \int_X \xi_{\rm ac}\wedge\chi^{n-1},
\end{equation}
and
\begin{equation}\label{2}
(\xi_{\rm ac})^n\geq f\omega^n,
\end{equation}
Lebesgue a.e. Then there is a closed positive current $R\in[\psi]$ with
\begin{equation}\label{3}
(R_{\rm ac})^n\geq f\omega^n,
\end{equation}
Lebesgue a.e.
\end{proposition}
\begin{proof}
Let $\mathcal{C}$ be the set of real positive $(1,1)$-currents $T\geq 0$ which satisfy $(T_{\rm ac})^n\geq f\omega^n$ Lebesgue a.e., which is nonempty (it contains $A\omega$ for $A\gg 1$). The set $\mathcal{C}$ is convex, thanks to the linearity of the Lebesgue decomposition of real positive $(1,1)$-currents, and the concavity of the map
\begin{equation}
T(x)\mapsto \left(\frac{(T_{\rm ac})^n}{\omega^n}(x)\right)^{\frac{1}{n}},
\end{equation}
for Lebesgue a.e. $x$, which is proved in \cite[Proof of Prop. 4.5]{Bou}.

We claim that $\mathcal{C}$ is closed in the weak topology. Since all currents in $\mathcal{C}$ are positive, and the weak topology is metrizable when restricted to real positive $(1,1)$-currents (see e.g. \cite[p.174]{DemB}), it suffices to check that $\mathcal{C}$ is weakly sequentially closed. This in turn follows immediately from \cite[Prop. 2.1]{Bou}, where it is proved that if positive real $(1,1)$-currents $T_i$ converge weakly to $T$, then for Lebesgue a.e. $x\in X$ we have
\begin{equation}
T_{\rm ac}^n(x)\geq \liminf_{i\to\infty} (T_i)_{\rm ac}^n(x).
\end{equation}
Note also that if $T\in\mathcal{C}$ then $T_{\rm ac}\in\mathcal{C}$ too.

If we let $\mathcal{F}$ be the set of closed real $(1,1)$-currents of the form $\psi+\ddbar u$, where $u$ is a distribution, then the conclusion of the Proposition is precisely the statement that $\mathcal{C}\cap\mathcal{F}$ is nonempty. The affine subspace $\mathcal{F}$ is also closed in the weak topology, since by the $\de\db$-Lemma it is identified with the space of closed real $(1,1)$-currents in the class $[\psi]$ (see also \cite[p.149]{BHPV} for another proof).

Suppose for a contradiction that $\mathcal{C}\cap\mathcal{F}=\emptyset,$ and consider the difference set $\mathcal{D}:=\mathcal{C}-\mathcal{F}=\{T-S\ |\ T\in\mathcal{C}, S\in\mathcal{F}\}$. It is convex, and by assumption it does not contain $0$ (the zero current). It is also weakly closed, by the following argument. Pick an arbitrary current $W\in \ov{\mathcal{C}-\mathcal{F}}$, and let $\mathcal{U}$ be the set of real $(1,1)$-currents $U$ with $\int_X U\wedge\omega^{n-1}<\int_X W\wedge\omega^{n-1}+1$, which is an open neighborhood of $W$ in the weak topology. Given any current $U\in \mathcal{U}\cap(\mathcal{C}-\mathcal{F})$, write $U=T-S$ with $T\in\mathcal{C}, S\in\mathcal{F}$, and observe that the $\omega$-mass of $T$ satisfies
\begin{equation}
\int_X T\wedge\omega^{n-1}=\int_X U\wedge\omega^{n-1}+\int_X S\wedge\omega^{n-1}\leq \int_X W\wedge\omega^{n-1}+\int_X\psi\wedge\omega^{n-1}+1=:M,
\end{equation}
where $M$ is independent of $U$. Thus, $T$ belongs to the weakly compact set $\mathcal{K}$ of currents in $\mathcal{C}$ (which is weakly closed) with mass bounded above by $M$, and so $U$ belongs to the difference set $\mathcal{K}-\mathcal{F}$, which is weakly closed since it is the difference of a weakly compact and a weakly closed set (and the weak topology is Hausdorff). This discussion shows that $\mathcal{U}\cap(\mathcal{C}-\mathcal{F})\subset  \mathcal{K}-\mathcal{F}$. Returning to our current $W$, it thus belongs to
\begin{equation}
\mathcal{U}\cap(\ov{\mathcal{C}-\mathcal{F}})\subset \ov{\mathcal{U}\cap(\mathcal{C}-\mathcal{F})}\subset\ov{ \mathcal{K}-\mathcal{F}}= \mathcal{K}-\mathcal{F}\subset
 \mathcal{C}-\mathcal{F},
\end{equation}
where the first inclusion uses that $\mathcal{U}$ is open, and we are done.

We can then apply the Hahn-Banach strict separation theorem to separate $\mathcal{D}$ and $\{0\}$, and we obtain a smooth real $(n-1,n-1)$-form $\sigma$ and a constant $c>0$ such that
\begin{equation}\label{hbb}
\int_X(T-S)\wedge\sigma\geq c,
\end{equation}
for all $T\in\mathcal{C}$ and $S\in\mathcal{F}$. This implies that $\sigma$ is $\de\db$-closed and positive (in the usual sense of positivity as currents, which for smooth $(n-1,n-1)$-forms is equivalent to being semipositive definite). For the first property, take any $T\in\mathcal{C}$ and any distribution $u$. Then for any $t\in\mathbb{R}$ we have that $\psi+t\ddbar u\in\mathcal{F}$, and so
\begin{equation}
\int_X (T-\psi)\wedge\sigma-t\int_X\ddbar u\wedge \sigma\geq c,
\end{equation}
and since this holds for all $t\in\mathbb{R}$ this forces $\int_X\ddbar u\wedge \sigma=0$. Since this holds for $u$ arbitrary, this means that $\de\db\sigma=0$.  To prove that $\sigma$ is positive, we need to show that
$\int_X V\wedge\sigma\geq 0$ for any positive real $(1,1)$-current $V$. Pick any $T\in\mathcal{C}$ and $t\in\mathbb{R}_{\geq 0}$. Then $T+tV$ is a positive real $(1,1)$-current with $(T_{\rm ac}+tV_{\rm ac})^n\geq (T_{\rm ac})^n$ Lebesgue a.e., and so $T+tV\in \mathcal{C}$. Applying \eqref{hbb} with $S=\psi$ gives
\begin{equation}
\int_X(T-\psi)\wedge\sigma+t\int_X V\wedge\sigma\geq c,
\end{equation}
and since this holds for all $t\geq 0$, we must have $\int_X V\wedge\sigma\geq 0$ as desired.

For any $\ve>0$, $\sigma+\ve\omega^{n-1}$ is therefore a $\de\db$-closed and strictly positive definite real $(n-1,n-1)$-form, which for any $T\in\mathcal{C}$ satisfies
\begin{equation}\label{hbb2}\begin{split}
\int_X(T-\psi)\wedge(\sigma+\ve\omega^{n-1})&\geq c+\ve\int_X T\wedge\omega^{n-1}-\ve\int_X\psi\wedge\omega^{n-1}\\
&\geq c-\ve\int_X\psi\wedge\omega^{n-1}\geq\frac{c}{2}>0,
\end{split}\end{equation}
provided we choose $\ve$ sufficiently small (independently of $T$). Fix this value of $\ve$. Then, as observed in \cite{Mi}, there is a Gauduchon metric $\chi$ with $\chi^{n-1}=\sigma+\ve\omega^{n-1}.$ By our assumptions \eqref{1} and \eqref{2}, there is a current $\xi\in\mathcal{C}$ which also satisfies \eqref{1}, hence contrasting this with \eqref{hbb2} (with $T=\xi_{\rm ac}\in\mathcal{C}$) we get
\begin{equation}
\int_X \psi\wedge\chi^{n-1}\geq \int_X\xi_{\rm ac}\wedge\chi^{n-1}\geq \int_X \psi\wedge\chi^{n-1}+\frac{c}{2},
\end{equation}
which is a contradiction.
\end{proof}
\begin{remark}
Proposition \ref{hb}, formulated using Bott-Chern cohomology, remains true also when $X$ is a general compact complex manifold, with very minor changes to the proof (using a Gauduchon metric instead of a K\"ahler one). Since we will not need this, we leave this to the interested reader.
\end{remark}

\subsection{The main proof}
\begin{proof}[Proof of Theorem \ref{main}]
Up to adding a small multiple of $[\omega]$ to $[\alpha]$ and $[\beta]$, and using that the RHS of \eqref{vol} is continuous as we vary the cohomology classes, we may assume without loss that $[\alpha],[\beta]$ are K\"ahler classes, and we fix K\"ahler forms $\gamma\in [\alpha], \eta\in [\beta]$. Let
\begin{equation}\label{del}
\delta:=\frac{n\int_X\gamma^{n-1}\wedge\eta}{\int_X\gamma^n}>0.
\end{equation}
We may assume without loss that $\delta<1$, otherwise \eqref{vol} is trivially true. Applying Proposition \ref{conc} to $(X,\eta)$, we obtain a closed positive current $T\in[\beta]$ with $\mathrm{Supp}(T)\subset\Sigma$, where $\Sigma$ is a closed Lebesgue nullset.

Write $T=\eta+\ddbar\vp$ for some quasi-psh function $\vp$.
Since $[\beta]$ is K\"ahler, applying Demailly's regularization \cite{Dem}, specifically the version in \cite[Thm 1]{BK}, produces a sequence $\eta+\ddbar\vp_k$ of smooth closed semipositive real $(1,1)$-forms cohomologous to $[\beta]$, with $\vp_k\downarrow\vp$. Then
\begin{equation}
T_k:=\eta+\ddbar\left(\frac{\vp_k}{1+\frac{1}{k}}\right)\geq \frac{1}{k+1}\eta>0,
\end{equation}
is a sequence of K\"ahler metrics in the class $[\beta]$ which converge weakly to $T$.

Choose open subsets $\{U_j\}_{j=1}^\infty$ of $X$ with $\ov{U_{j+1}}\subset U_j$ for all $j\geq 1$, and $\Sigma=\bigcap_{j=1}^\infty U_j$, and pick smooth cutoff functions $0\leq \rho_j\leq 1$ with $\mathrm{Supp}(\rho_j)\subset U_j$ and $\rho_j|_{U_{j+1}}\equiv 1$.

Let $\chi$ be an arbitrary Gauduchon metric on $X$. Since the Radon measure $T\wedge \chi^{n-1}$ is supported on $\Sigma$, for every $j\geq 1$ we have
\begin{equation}
\int_X \rho_j T\wedge \chi^{n-1}=\int_X T\wedge\chi^{n-1}=\int_X \eta\wedge\chi^{n-1}>0.
\end{equation}
By weak convergence of $T_k\wedge\chi^{n-1}$ to $T\wedge\chi^{n-1}$, for each $j\geq 1$ we can find $k(j)$ such that
\begin{equation}\label{int1}
\int_X \rho_j T_{k(j)}\wedge \chi^{n-1}\geq \int_X \eta\wedge\chi^{n-1}-\frac{1}{j}.
\end{equation}
To simplify notation, we will write simply $T_j:=T_{k(j)}.$ Then \eqref{int1} implies in particular that
\begin{equation}\label{int2}
\int_{U_j} T_{j}\wedge \chi^{n-1}\geq \int_X \eta\wedge\chi^{n-1}-\frac{1}{j}.
\end{equation}
For $j$ sufficiently large, the RHS of \eqref{int2} is positive, hence
\begin{equation}\label{int2b}
\left(\int_{U_j} T_{j}\wedge \chi^{n-1}\right)^2\geq \left(\int_X \eta\wedge\chi^{n-1}\right)^2-\frac{C}{j},
\end{equation}
for some constant $C$ independent of $j$, which may change from line to line in the following.

For each $j\geq 1$, by the Calabi-Yau Theorem we can find a K\"ahler metric $\gamma_j\in[\alpha]$ that satisfies
\begin{equation}\label{ma}
\gamma_j^n=(1-\delta)\gamma^n+\delta\frac{\int_X\gamma^n}{\int_X \eta\wedge\chi^{n-1}}T_j\wedge\chi^{n-1}.
\end{equation}
Indeed, the RHS of \eqref{ma} integrates to $\int_X\gamma^n$.\\

\noindent
{\bf Key claim. }For $j$ sufficiently large we have
\begin{equation}\label{claim}
\int_{U_j}\gamma_j\wedge\chi^{n-1}\geq \int_X \eta\wedge\chi^{n-1}-\frac{C}{j},
\end{equation}
for some constant $C$ independent of $j$.\\

\noindent
The proof of the key claim follows the method of \cite[Appendix]{BDPP}, but crucially on the LHS of \eqref{claim} we are only integrating on the small open set $U_j$, while on the RHS we integrate on all of $X$. This is a reflection of the key property that $T$ is supported on the nullset $\Sigma$.

To prove the key claim, at each point $x\in X$, let $0<\lambda_1(x)\leq \cdots\leq\lambda_n(x)$ be the eigenvalues of $\gamma_j(x)$ with respect to $T_j(x)$, which are continuous functions of $x$. Then, dropping the point $x$ from the notation, we have the pointwise inequality $\gamma_j\geq \lambda_1 T_j$, while \eqref{ma} implies that
\begin{equation}
\prod_{p=1}^n\lambda_p\geq\delta\frac{\int_X\gamma^n}{\int_X \eta\wedge\chi^{n-1}}\frac{T_j\wedge\chi^{n-1}}{T_j^n},
\end{equation} so we can estimate
\begin{equation}\label{int3}
\int_{U_j}\gamma_j\wedge\chi^{n-1}\geq \int_{U_j}\lambda_1 T_j\wedge\chi^{n-1}\geq \delta\frac{\int_X\gamma^n}{\int_X \eta\wedge\chi^{n-1}}\int_{U_j}\frac{1}{\lambda_2\cdots\lambda_n}\left(\frac{T_j\wedge\chi^{n-1}}{T_j^n}\right)^2 T_j^n.
\end{equation}
To estimate the RHS of \eqref{int3} we use H\"older
\begin{equation}\label{int4}
\left(\int_{U_j}T_j\wedge\chi^{n-1}\right)^2\leq \left(\int_{U_j}\frac{1}{\lambda_2\cdots\lambda_n}\left(\frac{T_j\wedge\chi^{n-1}}{T_j^n}\right)^2 T_j^n\right)
\left(\int_{U_j}\lambda_2\cdots\lambda_n T_j^n\right),
\end{equation}
and the elementary inequality
\begin{equation}
\frac{n\gamma_j^{n-1}\wedge T_j}{T_j^n}=\sum_{i=1}^n\prod_{p\neq i}\lambda_p\geq \lambda_2\cdots\lambda_n,
\end{equation}
which after integration gives
\begin{equation}
\int_{U_j}\lambda_2\cdots\lambda_n T_j^n\leq n\int_{U_j}\gamma_j^{n-1}\wedge T_j\leq n\int_X\gamma_j^{n-1}\wedge T_j=n\int_X\gamma^{n-1}\wedge\eta.
\end{equation}
This, combined with \eqref{int3} and \eqref{int4}, and using also \eqref{del} and \eqref{int2b}, gives
\begin{equation}\begin{split}
\int_{U_j}\gamma_j\wedge\chi^{n-1}&\geq \delta\frac{\int_X\gamma^n}{\int_X \eta\wedge\chi^{n-1}}\left(\int_{U_j}T_j\wedge\chi^{n-1}\right)^2\frac{1}{n\int_X\gamma^{n-1}\wedge\eta}\\
&=\frac{1}{\int_X \eta\wedge\chi^{n-1}}\left(\int_{U_j}T_j\wedge\chi^{n-1}\right)^2\\
&\geq \int_X \eta\wedge\chi^{n-1}-\frac{C}{j},
\end{split}\end{equation}
for all $j$ large, where $C$ is independent of $j$. This proves the key claim.\\

\noindent
Next, by precompactness of currents in a fixed class, up to passing to a subsequence (that we will suppress from the notation) we have $\gamma_j\rightharpoonup \gamma_\infty,$ where $\gamma_\infty$ is a closed positive current in $[\alpha]$. Write $\gamma_\infty=(\gamma_\infty)_{\rm ac}+(\gamma_\infty)_{\rm sing}$ for its Lebesgue decomposition. For any given $j>k\geq 1$, thanks to the key claim \eqref{claim} we have
\begin{equation}
\int_X \rho_k \gamma_j\wedge\chi^{n-1}\geq \int_{U_j}\gamma_j\wedge\chi^{n-1}\geq  \int_X \eta\wedge\chi^{n-1}-\frac{C}{j},
\end{equation}
so letting $j\to\infty$ we get
\begin{equation}
\int_X \rho_k \gamma_\infty\wedge\chi^{n-1}\geq\int_X \eta\wedge\chi^{n-1}.
\end{equation}
Since $\rho_k$ converges pointwise to the characteristic function of $\Sigma$, we can let $k\to\infty$ and dominated convergence (with respect to the finite measure $\gamma_\infty\wedge\chi^{n-1}$) gives
\begin{equation}
\int_\Sigma \gamma_\infty\wedge\chi^{n-1}\geq\int_X \eta\wedge\chi^{n-1}.
\end{equation}
But since the positive measure $(\gamma_\infty)_{\rm ac}\wedge\chi^{n-1}$ is absolutely continuous with respect to Lebesgue, and $\Sigma$ is a nullset, this gives
\begin{equation}
\int_X (\gamma_\infty)_{\rm sing}\wedge\chi^{n-1}\geq \int_\Sigma (\gamma_\infty)_{\rm sing}\wedge\chi^{n-1}\geq\int_X \eta\wedge\chi^{n-1},
\end{equation}
and so
\begin{equation}\begin{split}\label{g1}
\int_X (\gamma-\eta)\wedge\chi^{n-1}&=\int_X\gamma_\infty\wedge\chi^{n-1}-\int_X\eta\wedge\chi^{n-1}\\
&=
\int_X (\gamma_\infty)_{\rm ac}\wedge\chi^{n-1}+\int_X (\gamma_\infty)_{\rm sing}\wedge\chi^{n-1}-\int_X\eta\wedge\chi^{n-1}\\
&\geq \int_X (\gamma_\infty)_{\rm ac}\wedge\chi^{n-1}.
\end{split}
\end{equation}
This is where the crucial ``subtraction of $[\beta]$'' happened.
On the other hand, \eqref{ma} implies that
\begin{equation}
\gamma_j^n\geq(1-\delta)\gamma^n,
\end{equation}
and \cite[Prop. 2.1]{Bou} gives that
\begin{equation}\label{g2}
(\gamma_\infty)_{\rm ac}^n\geq(1-\delta)\gamma^n,
\end{equation}
Lebesgue a.e. We can then feed \eqref{g1} and \eqref{g2} into Proposition \ref{hb} (taking $\psi=\gamma-\eta$, $f=(1-\delta)\frac{\gamma^n}{\omega^n}$, and for each given Gauduchon metric $\chi$ we take $\xi=\gamma_\infty$) to get a closed positive current $R\in[\alpha]-[\beta]$ with
\begin{equation}
(R_{\rm ac})^n\geq(1-\delta)\gamma^n.
\end{equation}
Thus,
\begin{equation}
\vol([\alpha]-[\beta])\geq \int_X(R_{\rm ac})^n\geq(1-\delta)\int_X\gamma^n=\int_X\alpha^n-n\int_X\alpha^{n-1}\wedge\beta,
\end{equation}
and we are done.
\end{proof}

\end{document}